\documentclass[12pt]{article} 
\usepackage[utf8x]{inputenc}
\usepackage{tikz} 
\usepackage[all]{xy}
\usepackage{authblk}
\usepackage[babel=true]{csquotes}
\usepackage{amsmath,amsthm, amssymb,amsfonts}
\usepackage[hmargin=1in,vmargin=1in]{geometry}
\numberwithin{equation}{subsection} 
\xyoption{arc}
\usepackage{eucal}
\usepackage{indentfirst}
\usepackage{mathrsfs}
\usepackage{verbatim}
\usepackage{makeidx}
\usepackage{graphicx}
\usepackage{yfonts} 
\usepackage{hyperref}
\usepackage{enumitem} 
\usepackage{extpfeil} 
\usepackage{dsfont}
\usepackage[T1]{fontenc}
\newtheorem{thm}{Theorem}[section]

\newtheorem{lem}[thm]{Lemma}

\newtheorem{cor}[thm]{Corollary}

\newtheoremstyle{bidule}
{3pt}
{3pt}
{}
{}
{\scshape}
{.}
{.5em}
{}
\newtheorem{df}[thm]{Definition}

\theoremstyle{definition}
\newtheorem{rmk}[thm]{Remark}

\newcommand{\E}{\mathscr{E}}

\newcommand{\C}{\mathcal{C}}

\newcommand{\Ar}{\text{Arr}}

\newcommand{\D}{\mathcal{D}}

\newcommand{\A}{\mathcal{A}}
\newcommand{\Aa}{\mathcal{A}}
\newcommand{\B}{\mathscr{B}}

\newcommand{\J}{\mathcal{J}} 

\newcommand{\R}{\mathbb{R}}

\newcommand{\Cx}{\mathbb{C}}

\renewcommand{\to}{\longrightarrow}
\newcommand{\ol}{\overline}

\newcommand{\Ob}{\text{Ob}}
\renewcommand{\1}{\textbf{1}} 

\newcommand{\tx}{\text}

\renewcommand{\to}{\longrightarrow}
\DeclareMathOperator\Id{Id}
\DeclareMathOperator\Hom{Hom}

\DeclareMathOperator\Set{\textbf{Set}} 
\DeclareMathOperator\Cat{\mathbf{Cat}}
\DeclareMathOperator\colim{\tx{$colim$}}
\DeclareMathOperator\Vect{\textbf{Vect}} 

\DeclareMathOperator\Top{\mathbf{Top}}
\DeclareMathOperator\Var{\mathbf{Var}}

\title{Displacements} 
\author{Hugo V. Bacard \thanks{\textit{E-mail address}: \href{mailto:hbacard@uwo.ca}{hbacard@uwo.ca}
}}
 \affil{Western University}
\date{\today}
\begin{document}
\maketitle
\begin{abstract}
Given a functor $p:\E \to \B$ and an object $e \in \E$ , we define a \emph{displacement} of $e$ along a morphism $\varepsilon: p(e) \to b$, as a map $e \to \nabla_\varepsilon(e)$ satisfying a universal property analogue to that of a \emph{cocartesian lifting} (pushforward) \emph{à la} Bénabou-Grothendieck-Street. There are many difficulties in geometry that come from the fact that forgetful functors such as $p: \Var(\Cx) \to \Top $ don't have displacements of objects along arbitrary maps.  And this can be already seen abstractly, since the existence of a left adjoint to $p$, can be reduced to the existence of all displacements of the initial object. However some \emph{schematization functors} exist as approximations. In a broader context, if $\B$ is a model category and $p$ is a right adjoint, then the right-induced model category on $\E$ exists if and only if all displacements along any trivial cofibration $\varepsilon$, are weak $p$-equivalences. In these notes we provide some categorical lemmas that will be necessary for future applications. The idea is to have a \emph{homotopy descent process} for \emph{elementary displacements} when $p$ has a \emph{presentation} as a $2$-pullback of a family $\{p_i: \E_i \to \B\}_{i\in J}$. When suitably applied it should lead to techniques similar to Mumford's GIT through homotopy theory (simplicial presheaves).
\end{abstract}
\setcounter{tocdepth}{1}

\section{Motivations}
The idea of a displacement is really to move an object equipped with a \emph{structure} along an underlying map. For example if we are given a vector space $V$ and a bijection(=symmetry)  of sets $\varepsilon:p(V) \to W$, then it's easily seen that we can turn $W$ into a vector space and $\varepsilon$ lifts to an isomorphism of vector spaces. But if $\varepsilon$ is not an isomorphism, things get complicated. Solving this problem is equivalent to determine the existence of displacements, for the forgetful functor $p: \Vect \to \Set$.\\

A more interesting example is to consider the forgetful functor $p: \Var(\Cx) \to \Top$.  Because in this case, given an algebraic variety $X$ and a homeomorphism  of its underlying space of complex points $\varepsilon: X(\Cx) \to X(\Cx)$, then $\varepsilon$ hardly lifts to an algebraic map $X \to X$. Having a homeomorphism of underlying topological spaces is having a topological symmetry. But we conjectured in \cite{Bacard_Sym_QS2}, that the ``raison d'être'' of the field $\Cx$ of complex numbers, is precisely the presence of the Higgs boson. And if we are given a symmetry that forgets the complex structure, we are given a symmetry that forgets the Higgs field. And therefore it's not surprising that we cannot always lift the non-Higgs symmetry to a symmetry between complex structures (Higgs symmetry).\\

However if we are given a symmetry, or in general a morphism, that remembers the presence of the Higgs field, which is the complex structure, then  we should theoretically expect to have a lift, and this is confirmed by Serre's GAGA principle \cite{Serre_GAGA}.\\

Our motivation is that things don't work because ``they work'', but there should be a (hidden) reason why things are working. And as for complex structures, we conjectured that the main reason is the presence of the Higgs field. And as long as we acknowledge its presence by doing operations that preserve it, ``things should work''.\\

There are many concepts that are used in algebraic geometry such as \emph{Lefschetz principle} or even classical \emph{Spectral sequences}, that we always found as `magic' or mysterious. We think that the explanation for these mysteries is in set theory and its problems. The foundations of current maths are based on set theory and \emph{sets don't have symmetries} as already envisioned by Grothendieck. This fact was our motivation for writing \cite{Bacard_Sym_QS2}. For example Lefschetz principle is another way of saying that every algebraic closed field of characteristic $0$ is in the \emph{connected component} of $\Cx$. And it's reasonable to think that there is a deep meaning of this fact, if we use sets with symmetries and write $\pi_0(\Cx), \pi_0(\R)$, etc, instead. The basic constructions in most \emph{Geometries} is to build \emph{spaces} from \emph{contractible ones} (affines), by \emph{gluing or descending} them. But for a long time, we did not use the fact that \emph{contractible} \textbf{is} \emph{homotopy theory}. \\ 

Another example of displacement comes with the functor $H^0: dg\tx{-}\Cat \to \Cat.$ To\"en \cite{Toen_Morita} considered the analogue of a Gabriel-Zisman localization for a dg-category $\C$ along a subset of maps $S \subset H^0(\C)$. And this can be seen as a displacement along the usual localization functor $L: H^0(\C) \to S^{-1} H^0(C)$.\\

Similar situations occur when we modify locally the equations of an algebraic variety by some group action (symmetry). In general when we do this we usually jump from a category of rigid structures such as schemes to a much flexible category near schemes, such as \emph{(pre)stacks}.\

In algebraic topology, and therefore  in higher category, studying displacements for some functors from $(n+1)$-homotopy types to $n$-homotopy types is far from being trivial. In fact, I  remember hearing Mark Behrens who said in a talk, something like:
\begin{center}
\emph{``You think that you know about the identity functor until you start doing Goodwillie Calculus...''}
\end{center} 

And there is no reason to think he's mistaking. That is to say, studying displacements in general can be complicated.
\newpage
\tableofcontents
\ \\
 
 \ \\
 
 \ \\
\ \\

\ \\ 
 
\begin{center}
\begin{quote}
``... Je m’y adresse à toi qui me lis comme à une \textbf{personne}, et à une personne \textbf{seule}. C’est à celui en
toi qui sait être seul, à l’enfant, que je voudrais parler, et à personne d’autre. Il est loin souvent l’enfant, je
le sais bien. Il en a vu de toutes les couleurs et depuis belle lurette. Il s’est planqué Dieu sait où, et c’est pas
facile, souvent, d’arriver jusqu’à lui. On jurerait qu’il est mort depuis toujours, qu’il n’a jamais existé plutôt -
et pourtant, je suis sûr qu’il est là quelque part, et bien en vie.

Et je sais aussi quel est le \textbf{signe} que je suis entendu. C’est quand, au delà de toutes les différences de culture et de destin, ce que je dis de ma personne et de ma vie trouve en toi écho et résonance ; quand tu y
retrouves aussi ta propre vie, ta propre expérience de toi-même, sous un jour peut-être auquel tu n’avais pas
accordé attention jusque là.''

\hfill \emph{L'importance d'être seul}, Récoltes et Semailles \

\hfill Alexander Grothendieck

\end{quote}
\end{center}
\newpage
\section{Definition and properties}
\subsection{Displacement of an object}\label{section-displ}
The following definition is weaker notion of a \emph{Street opfibration}. We make no claim of introducing this definition\footnote{After all it's just a definition} .
\begin{df}
Let $p: \E \to \B$ be a functor and let $e$ be an object of $\E$. Let $\varepsilon: p(e) \to b$ be a morphism in $\B$ and let $p^{\ast}$ be the induced functor between the comma categories:
$$ p^{\ast}:(e \downarrow \E) \to (p(e) \downarrow \B).$$

A \emph{displacement of $e$ along $p(e) \xrightarrow{\varepsilon} b$} is an object $e \to \nabla_\varepsilon(e)$ of $(e \downarrow \E)$ that corepresents the functor
$$\Hom(\varepsilon,-) : (e \downarrow \E) \to \Set.$$
This functor takes $h:e \to e'$ to the hom-set $\Hom(\varepsilon, p(h))$ of $(p(e)\downarrow \B)$. In other words a displacement is just an \emph{adjoint-transpose} of $\varepsilon$ through $p^\star$.\\ 

In particular a displacement along any $\varepsilon: p(e) \to b$ exists if and only if $p^{\ast}$ has a left adjoint.
\end{df}

\begin{rmk}\label{rmk-ppty-disp}
\begin{enumerate}
\item If a displacement $e \to \nabla_\varepsilon(e)$  along $p(e) \xrightarrow{\varepsilon} b$ exists, then there is a universal map $b \to p(\nabla_\varepsilon(e))$ (the unit of the adjunction), such that the map $p(e \to \nabla_\varepsilon(e))$ is the following composite.
\begin{equation}\label{univ-factor}
p(e) \to b \to p(\nabla_\varepsilon(e)).
\end{equation}
\item If for every $e$, there is a displacement along any $p(e) \to b$, such that the universal map $b \to p(\nabla_\varepsilon(e))$ is an isomorphism; then $p$ is a Street opfibration. In that case the map $e \to \nabla_\varepsilon(e)$ is a pseudo cocartesian lifting of $\varepsilon:p(e) \to b$.
\item If both $\E$ and $\B$ have initial objects  $e_0$ and $b_0$, respectively; it's well known that we have an equivalence of categories $(e_0 \downarrow \E)  \simeq  \E$ and and equivalence $(b_0 \downarrow \B ) \simeq  \B$. Therefore  if $p(e_0)=b_0$ then $p$ has a left adjoint if and only if all displacements of $e_0$ exists.
\item Let $\1=\{0, \Id_0\}$ be the unit category and let $\1< \E$ be the \emph{join category}: there is exactly one morphism from $0$ to any object in $\E$ and no morphism whose target is $0$ except the identity $\Id_0$. So roughly speaking we're adding an initial object $0$. Any functor $p: \E \to \B$ induces a functor $$(\1 < p): (\1< \E) \to (\1< \B)$$
that restricts to $p$ on $\E$. 

It's not hard to see that $p$ has a left adjoint if and only if $(\1<p)$ has one. This breaks down problems on an existence of an adjoint to problems on existence of displacements of initial object. 
\end{enumerate}
\end{rmk}

\subsection{ Join constructions and Pseudopullbacks} 
Let $\{p_j: \E_j \to \B \}_{j \in \J}$ be a family of functors over the same base $\B$.  Let $\E= \times_{\B} \E_j$ be a \emph{pseudopullback} (also called  $2$-pullback) of this family in $\Cat$ and let $\tau_j: \E \to \E_j$ be the canonical projection.  We refer the reader to \cite{Joyal_Street_pspb} for the definition of a pseudopullback. We remind the reader that we can take as model for $\E$ the category described as follows.
\begin{enumerate}
\item The objects of $\E$ are cone of isomorphisms $\{b \xrightarrow{\cong} p_j(e_j); e_j \in \E_j  \}_{j \in \J}$;
\item A morphism $\sigma:\{b \xrightarrow{\cong} p_j(e_j); e_j \in \E_j  \}_{j \in \J} \to \{c \xrightarrow{\cong} p_j(f_j); f_j \in \E_j  \}_{j \in \J} $
consists of a  morphism $\sigma:b \to c$ and a family of morphisms $\{\sigma_j; e_j \to f_j \}_{j \in \J}$ such that for each $j$ the following commutes.

\[
\xy
(0,20)*+{b}="X";
(20,20)*+{p_j(e_j)}="Y";
(0,0)*+{c}="A";
(20,0)*+{p_j(f_j)}="B";
{\ar@{->}^-{\cong}"X";"Y"};
{\ar@{->}_-{\cong}"A";"B"};
{\ar@{->}^-{p_j(\sigma_j)}"Y";"B"};
{\ar@{->}_-{\sigma}"X";"A"};
\endxy
 \]
\end{enumerate}
In general there is no canonical map $p: \E \to \B$ but a family of naturally isomorphic functors  $p_j \circ \tau_j$. We will assume that a choice $p: \E \to \B$ has been made once and for all. The advantage of working with the above model is that there is a canonical projection $p: \E \to \B$ that takes $\{b \xrightarrow{\cong} p_j(e_j); e_j \in \E_j  \}_{j \in \J}$ to $b$.

\begin{lem}
Let $\A$ be a category and let $(1<\B)$ be the join category described above. Then a functor $F:\A \to (1< \B)$ is completely determined by the following data.
\begin{itemize}
\item Two full subcategories $\A_{-}$ and $\A_+$ of $\A$ with $\Ob(\A)= \Ob(\A_-) \sqcup \Ob(\A_+)$; and such that there is no morphism $a_+ \to a_-$ in $\A$ with $a_+ \in \A_+$ and $a_- \in \A_-$.
\item A functor $F_+:\A_+ \to \B$ such that the diagram hereafter is a pseudopullback.
\[
\xy
(0,20)*+{\A_+}="A";
(20,20)*+{\A}="B";
(0,0)*+{\B}="C";
(20,0)*+{(1<\B)}="D";
{\ar@{->}^{}"A";"B"};
{\ar@{->}_{F_+}"A";"C"};
{\ar@{->}^{F}"B";"D"};
{\ar@{->}^{}"C";"D"};
\endxy
\]  
\end{itemize} 
\end{lem}

\begin{proof}
Given $F:\A \to (1< \B)$, we let $\A_-$ be the full subcategory whose objects are the elements of $F^{-1}(0)$. And we let $\A_+$ be the full subcategory of $\A$ whose set of objects is the complementary of $\Ob(\A_-)$ in $\Ob(\A)$. By construction we have  $\Ob(\A)= \Ob(\A_-) \sqcup \Ob(\A_+)$. Furthermore since there are no morphisms in $\B$ whose target is $0$ except $\Id_0$, we cannot have a map $a_+ \to a_-$ because then we will have a function $\A(a_+,a_-) \to \emptyset$ whose target is the empty set but the source is not, which is impossible.\\

Conversely given $, \A_-,\A_+$ and $F_+$ we define $F$ as follows. $F$ is constant of value $0$ on $\A_-$ and is equal to $F_+$ on $\A_+$. For any map $a_- \to a_+$ in $\A$ we let $F(a_- \to a_+)$ be the unique map $0 \to F_+(a_+)$ in $(1<\B)$. One can check that this defines indeed a functor $F: \A \to (1<\B)$ whose restriction to $\B$ is (by construction)  $F_+$.
\end{proof}

A direct consequence of the lemma is:
\begin{cor}\label{cor-join-pseudo}
The functor $(1<-): \Cat \to \Cat$ preserves pseudopullbacks. 
\end{cor}
\begin{rmk}
\begin{enumerate}
\item According to the notation of the lemma, if we take $\A= (1< \B)$ and $F= \Id$ then we may write $\B= (1< \B)_+$ and $\1= (1< \B)_-$.
\item When we have such functor $F:\A \to (1< \B)$ , we will say that $F$ is a \emph{one way bridge from $\A_-$ to $\A_+$}.
\item Observe that the homset $\A(a_-,a_+)$ defines a bimodule $\A_- \times \A_+^{op} \to \Set$.
\end{enumerate}
\end{rmk}

\section{Left perfectness and Descent for displacements}
\begin{df}
Let $\B$ be a category containing two classes of morphisms called \emph{cofibrations} and \emph{trivial cofibrations}, each of them closed under transfinite composition and cobase change. Let  $p: \E \to \B$ be as above and let $e$ be an object of $\E$.

\begin{enumerate}
\item Say that $p$ is \emph{left perfect at $e$} if for any (trivial) cofibration $\varepsilon: p(e) \to b$ in $\B$, the universal map $b \to p(\nabla_\varepsilon(e))$ is also a (trivial) cofibration for a displacement $e \to \nabla_\varepsilon(e) $ of $e$ along $\varepsilon$.
\item Say that $p$ is left perfect if it's left perfect at any object $e \in \E$.
\item Say that a map $f:e \to e'$ in $\E$ is a (trivial) $p$-cofibration if $p(f)$ is a (trivial) cofibration in $\B$
\end{enumerate} 

\end{df}

\begin{rmk}
Thanks to the universal factorization \eqref{univ-factor} in Remark \ref{rmk-ppty-disp}, it's not hard to see that  if $p$ is left perfect at $e$, then for any displacement  $\eta: e \to \nabla_\varepsilon(e)$, $p(\eta)$ is a (trivial) cofibration if $\varepsilon$ is a (trivial) cofibration. In other words $\eta$ is a (trivial) $p$-cofibration if $\varepsilon$ is a (trivial) cofibration in $\B$.
\end{rmk}
\paragraph{Crossing Lemma}
Let $\B$ be a category and let $\lambda$ and $\kappa$ be two infinite regular cardinals with $\lambda < \kappa$. Assume that $\B$ has all $\kappa$-small colimits \footnote{ In most cases we will assume also that $\B$ is locally $\lambda$-presentable (hence locally $\kappa$-presentable)}. Let's start with the following lemma which is a tautology. We mention it because it appears many times in the upcoming constructions.
\begin{lem}[\emph{Crossing lemma}]\label{cross-lem}
Let $C: \lambda \to \B$ and $D: \lambda \to \B$ be two directed diagrams in $\B$. Assume that for every $k \in \lambda$ there exists two maps $\eta_k: C_k \to D_k$ and  $\varepsilon_k: D_k \to C_{k+1}$ such that the structure maps of $C$ and $D$ are respectively the composite below.
 $$C_k \to C_{k+1}= C_k \xrightarrow{\eta_k} D_k \xrightarrow{\varepsilon_k} C_{k+1}$$ 
 $$D_k \to D_{k+1}= D_k \xrightarrow{\varepsilon_k} C_{k+1} \xrightarrow{\eta_{k+1}} D_{k+1}$$

Then $C$ and $D$ have isomorphic colimits and the maps between the colimits that are induced by $\varepsilon_k$ and $\eta_k$ are inverse each other.
\end{lem}
\begin{proof}
Let $C_\infty$ and $D_\infty$ be the corresponding colimits and let $i_k: C_k \to C_\infty$ and $j_k: D_k \to D_\infty$ be the canonical maps. Denote by $\varepsilon_\infty: D_\infty \to C_\infty$ and $\eta_\infty: C_\infty \to D_\infty$ the universal maps induced by the maps $\varepsilon_k$ and $\eta_k$.\\

We have the following equality for each $k$.
$$D_k \xrightarrow{\varepsilon_{k}}C_{k+1} \xrightarrow{i_{k+1}} C_\infty= D_k\xrightarrow{j_k} D_\infty \xrightarrow{\varepsilon_\infty} C_\infty ;$$
$$C_k \xrightarrow{\eta_{k}}D_k \xrightarrow{j_k} D_\infty= C_k \xrightarrow{i_k} C_\infty \xrightarrow{\eta_\infty} D_\infty.$$

If we precompose by $\eta_k: C_k \to D_k$ in the first equality and then use the second equality we see that we have $i_k= (\varepsilon_\infty \circ \eta_\infty) \circ i_k$.  Similarly if we precompose by $\varepsilon_{k-1}: D_{k-1} \to C_k $ in the second equality and then use the first equality we get $j_{k-1}= (\eta_\infty \circ \varepsilon_\infty) \circ j_{k-1}$.\\

But on the other hand, by definition of a colimit, the only endomorphism $f \in \Hom(C_\infty, C_\infty)$ such that $f \circ i_k= i_k$ for all $k$ is the identity $\Id_{C_\infty}$. The same holds for $D_\infty$ with the maps $j_k$. This forces the two equalities $\eta_\infty \circ \varepsilon_\infty= \Id$ and $\varepsilon_\infty \circ \eta_\infty=\Id$ and the lemma follows.
\end{proof}

\begin{lem}\label{lem-displ-intersec}
Let $\kappa$ be a regular cardinal and let $\{p_j: \E_j \to \B \}_{j \in \J}$ be a $\kappa$-small family of functors over the same base $\B$.  Let $\E= \times_{\B} \E_j$ be a $2$-pullback of this family and let $p: \E \to \B$ be `the' canonical projection.

Assume that
\begin{itemize}
\item For every $e_j \in \E_j$ all displacements of $e_j$ exist;
\item Every $p_j$ creates (and hence preserves) filtered colimits in $\E_j$. 
\item $\B$ is closed under $\kappa$-small filtered colimits.
\item Assume furthermore that $\B$ is closed under $\kappa$-small wide pushouts.
\end{itemize}

Then for every $e\in \E$, all displacements of $e$ exist. Furthermore if $\B$ is a category with two classes of maps called cofibrations and trivial cofibrations each of them closed under transfinite composition and cobase change, and if each $p_j: \E_j \to \B$ is left perfect then $p: \E \to \B$ is left perfect.
\end{lem}
\begin{rmk}
In practice we will use the lemma  when $\B$ and every $\E_j$ have an initial object and every $p_j$ sends initial object to initial object.
\end{rmk}
\begin{proof}

Let $\tau_j: \E \to \E_j$, $j \in \J$ be the universal family of functors. Let $e$ be an object of $\E$ and let $\varepsilon: p(e) \to b$ be a morphism in $\B$. Recall that for every $j$ there is an isomorphism $p_j\tau_j \cong p$; in particular there is a morphism $\varepsilon_j: p_j \tau_j(e) \to b$ which is isomorphic to $\varepsilon$ as objects of $(\B \downarrow b)$.\\

Let $\lambda$ be another regular cardinal with $\lambda < \kappa$. We are going to construct inductively and simultaneously for all $j$, a family  $\lambda$-directed diagrams $\{ e_j^\bullet: \lambda \to \E_j\}_{j \in \J}$ and $b^\bullet: \lambda \to \B$.

\begin{enumerate}
\item Let $b^0=b$,  $e_j^0=\tau_j(e)$ and let $\varepsilon_j^0: p_j(e_j^0) \to b^0$ be the above map $\varepsilon_j: p_j \tau_j(e) \to b$ (that is isomorphic to $\varepsilon$).
\item For each $j$, we define the structure map $\eta_j^k:e_j^{k} \to  e_j^{k+1}$ of the diagram $e_j^\bullet: \lambda \to \E_j$ as the displacement of $e_j^{k}$ along the (already existing) map $$\varepsilon_j^{k}: p_j(e_j^k) \to b^k.$$
This means that  $e_j^{k+1}\cong \nabla_{\varepsilon_j^{k}}(e_j^k)$.
\item  Following Remark \ref{rmk-ppty-disp}, there is a universal map $\delta_j^k: b^k \to p_j(e_j^{k+1})$ such that we have an equality
$$p_j(e_j^{k}) \xrightarrow{p_j(\eta_j^k)}  p_j(e_j^{k+1})=   p_j(e^k) \xrightarrow{\varepsilon_j^{k}} b^k \xrightarrow{\delta_j^k} p_j(e_j^{k+1}).$$
\item Let $b^{k+1} \in \B$ be the colimit of the wide pushout data $\{b^k \xrightarrow{\delta_j^k} p_j(e_j^{k+1}) \}$. 
\item We define the structure map $\iota^k: b^k \to b^{k+1}$ of the diagram $b^\bullet: \lambda \to \B$ as the canonical map going to the colimit of the wide pushout data.
\item By construction, for every $j$ there is also a canonical map $\varepsilon_j^{k+1}: p_j(e_j^{k+1}) \to b^{k+1}$ and we have the following equality.
$$b^k \xrightarrow{\iota^k} b^{k+1}= b^k \xrightarrow{\delta_j^k} p_j(e_j^{k+1}) \xrightarrow{\varepsilon_j^{k+1}} b^{k+1}.$$
\item Let $b^{\infty}$ be the colimit of the $(b^k)$ and let $e_j^{\infty}$ be the colimit of the $(e_j^{k})$. Note that $e_j^{\infty}$ exists since $p_j$ creates filtered colimits and $\B$ is closed under filtered colimits.
\end{enumerate}
It's clear from the construction that for every $j$ the two directed diagrams $\{ b^k \xrightarrow{\iota^k} b^{k+1}\}$ and $\{p_j(e_j^{k}) \xrightarrow{p_j(\eta_j^k)}  p_j(e_j^{k+1}) \}$ are \emph{crossing}. It follows from our \emph{Crossing lemma} (Lemma \ref{cross-lem}) that they have isomorphic colimits i.e., 
$$\colim \{p_j(e_j^{k}) \xrightarrow{p_j(\eta_j^k)}  p_j(e_j^{k+1}) \} \cong b^{\infty}, \quad \forall j \in \J.$$

On the other hand, we know by assumptions that $p_j$ creates and thus preserves filtered colimits. It turns out that we have 
$$b^{\infty}\cong \colim \{p_j(e_j^{k}) \xrightarrow{p_j(\eta_j^k)}  p_j(e_j^{k+1}) \}  \cong p_j [ \colim \{e_j^{k}\xrightarrow{\eta_j^k} e_j^{k+1} \}]= p_j(e_j^{\infty}).$$  

It follows that for every $i, j \in \J$ we have $p_j(e_j^{\infty}) \cong p_i(e_i^{\infty})$ in $\B$. Note already that these isomorphisms determine an object in the pseudopullback.\\

Let us regard each canonical map $\tau_j(e) \to e_j^{\infty}$ as a functor $\alpha_j: [1] \to \E_j$. From the previous discussion we have a natural isomorphism 
$$p_j \alpha_j \cong p_i \alpha_i, \quad \forall i, j \in \J.$$
 
The universal property of the $2$-pullback implies that there exist a map \footnote{essentially unique map} $\alpha:[1] \to \E$ such that for every $j$ we have

\begin{equation}\label{eq_alpha}
\alpha_j\cong  \tau_j \alpha.
\end{equation} Let us regard $e \in \E$ as given by the family functor $e_j= \tau_j(e): \1  \to \E_j$ satisfying $ p_j(e_j) \cong p_i(e_i)$ (here $\1$ is the unit category).\\ 

Then $\alpha: [1] \to \E$ defines a map $e \to \nabla_\varepsilon(e)$ in $\E$ with $\nabla_\varepsilon(e)= \alpha(1)$. The isomorphism \eqref{eq_alpha} says that for every $j$  the morphism $\tau_j[e \to \nabla_\varepsilon(e)]$ is isomorphic in $\E_j^{[1]}$ to the morphism $\tau_j(e) \to e_j^{\infty}$. This implies in particular that for every $j$ there is an isomorphism $p_j[\nabla_\varepsilon(e)] \cong e_j^\infty$. Note that the universal map $b \to p[\nabla_\varepsilon(e)]$ is essentially (=isomorphic to) the map $b \to b^{\infty}$.

\paragraph{Checking the universal property} We are going to show that $\alpha: e \to \nabla_\varepsilon(e) $ satisfies the universal property of a displacement of $e$ along $\varepsilon$.\\

Let $h:e \to d$ be a morphism in $\E$ such that $p(h)$ factors through $\varepsilon: p(e) \to b$ as 
\begin{equation}\label{eq_h}
p(e) \xrightarrow{p(h)} p(d)= p(e) \xrightarrow{\varepsilon} b \xrightarrow{q} p(d);
\end{equation}
for some map $q: b \to p(d)$. We wish to show that there exists a unique map $\xi:  \nabla_\varepsilon(e) \to d $ such that $h= \xi \alpha$.

Recall that for every $j$, we have  $e_j^0= \tau_j(e) $, $b_0 = b$ and $\varepsilon_j^0: p_j (e_j^0) \to b_0$ is the map $\varepsilon$ precomposed with the isomorphism $p_j \tau_j(e) \cong p(e)$. Let $h_j: e_j^0 \to d_j$ be the image of $h$ by $\tau_j$. Thanks to the isomorphism $p \cong p_j \tau_j$ we have for every $j$:
\begin{equation}
p_j(h_j) \cong p(h). 
\end{equation}

It's not hard to see that from \eqref{eq_h} we have an equality for every $j$:
\begin{equation}
p_j(e_j^0) \xrightarrow{p_j(h_j)} p_j(d_j)= p_j(e_j^0) \xrightarrow{\varepsilon_j^0} b^0 \xrightarrow{q_j} p_j(d_j);
\end{equation}
where $q_j$ is the map $q$ composed with the isomorphism $p(d) \cong p_j(d_j)$. Let's denote by $\psi_j: p(d) \xrightarrow{\cong} p_j(d_j) $ this isomorphism so that $q= \psi_j^{-1}q_j$.
\subparagraph{Inductive factorization} For $k=0$ we have the following data. 
\begin{enumerate}
\item A map $h_j^k: e_j^k \to d_j$ in $\E_j$, for all $j$;
\item A map $q^k:b^{k}\to p(d)$ in $\B$;
\item A map $\varepsilon_j^k: p_j(e_j^k) \to b^k$; 
\item A map $q_j^k: b^k \to p_j(d_j) $ in $\B$ such that $q^k= \psi_j^{-1}q_j^{k}$, where $\psi_j: p(d) \xrightarrow{\cong} p_j(d_j)$ is a fixed isomorphism.
\item For every $j$ we have an equality 
\begin{equation}\label{eq-hj}
p_j(e_j^k) \xrightarrow{p_j(h_j^k)} p_j(d_j)= p_j(e_j^k) \xrightarrow{\varepsilon_j^k} b^k \xrightarrow{q_j^k} p_j(d_j);
\end{equation}
\end{enumerate}

We construct the data for $k+1$ as follows.\\

For every $j$, $\eta_j^k:e_j^k \to e_j^{k+1}$ is a displacement of $e_j^k$ along $\varepsilon_j^k$,  therefore with Equation \eqref{eq-hj}, the universal property of the displacement gives a unique map $h_j^{k+1}: e_j^{k+1} \to d_j$ such that the following equalities hold. 
\begin{equation}\label{eq-h1j}
h_j^k= h_j^{k+1} \eta_j^k
\end{equation}
\begin{equation}\label{eq-delj}
b^k \xrightarrow{q_j^k} p_j(d_j)= b^k \xrightarrow{\delta_j^k} p_j(e_j^{k+1}) \xrightarrow{p_j(h_j^{k+1})} p_j(d_j).
\end{equation}

Applying $\psi_j^{-1}$ to Equation \eqref{eq-delj} gives a factorization of $q^k$ for every $j$ as:
\begin{equation}\label{eq-q}
q^k= b^k \xrightarrow{q_j^k} p_j(d_j) \xrightarrow{\psi_j^{-1}} p(d)= b^k \xrightarrow{\delta_j^k} p_j(e_j^{k+1}) \xrightarrow{p_j(h_j^{k+1})} p_j(d_j) \xrightarrow{\psi_j^{-1}} p(d).
\end{equation}

Now $b^{k+1}$ together with the maps $\{p_j(e_j^{k+1}) \xrightarrow{\varepsilon_j^{k+1}} b^{k+1} \}$ is defined as the wide pushout of the maps $\{b^k \xrightarrow{\delta_j^k} p_j(e_j^{k+1}) \}_{j \in \J}$.  Therefore by \eqref{eq-q} there exists a unique map $q^{k+1}: b^{k+1} \to p(d)$ such that the equalities below hold.
\begin{equation}
b^k \xrightarrow{q^k} p(d)= b^k \xrightarrow{\iota^k} b^{k+1} \xrightarrow{q^{k+1}} p(d) \quad \tx{i.e.} \quad  q^k= q^{k+1} \iota^k;
\end{equation}
\begin{equation}\label{eq-psij}
 p_j(e_j^{k+1}) \xrightarrow{p_j(h_j^{k+1})} p_j(d_j) \xrightarrow{\psi_j^{-1}} p(d)=  p_j(e_j^{k+1}) \xrightarrow{\varepsilon_j^{k+1}} b^{k+1} \xrightarrow{q^{k+1}}p(d).
\end{equation}

If we let $q_j^{k+1}= \psi_j q^{k+1}$, and compose with $\psi_j$ in the equality \eqref{eq-psij} we get:
\begin{equation}
 p_j(e_j^{k+1}) \xrightarrow{p_j(h_j^{k+1})} p_j(d_j) =  p_j(e_j^{k+1}) \xrightarrow{\varepsilon_j^{k+1}} b^{k+1} \xrightarrow{q_j^{k+1}}p_j(d_j).
\end{equation}

The above maps and equations give the data for $k+1$. And by induction, we see that the relations \eqref{eq-h1j} determine a compatible diagram ending at $d_j$. Therefore from the universal property of $e_j^{\infty}$ there is a unique map $h_j^{\infty}: e_j^{\infty} \to d_j$ such that for every $k$ we have $h_j^k= h_j^{\infty} \circ \eta_j^k$. 

In particular for $k=0$ we get:
\begin{equation}
e_j \xrightarrow{h_j} d_j = e_j \xrightarrow{\alpha_j}e_j^{\infty} \xrightarrow{h_j^{\infty}} d_j
\end{equation}

The family $\{e_j \xrightarrow{h_j} d_j\}_{j \in \J}$ determines a morphism in the pullback $\E$, that is unique up-to an isomorphism in $\E^{[1]}$. And a morphism in a pseudopullback is unique if we fix the source and target and the comparison maps $p_j(e_j) \cong p_i(e_i)$, $p_j(d_j) \cong p_i(d_i)$. This means that $h: e \to d$ is the unique morphism in the pullback whose source is $e$ and target is $d$ and such that for every $j$: $\tau_j(h)=h_j$.\\

Similarly $\{e_j \xrightarrow{\alpha_j}e_j^{\infty} \}_{j \in \J}$ determine our map $\alpha: e \to \nabla_\varepsilon(e)$ and $\{e_j^{\infty}\xrightarrow{\alpha_j} d_j \}_{j \in \J}$ determine  uniquely a map $\xi: \nabla_\varepsilon(e) \to d$. Note that by construction we have a comparison isomorphism $e_j^{\infty} \xrightarrow[s_j]{\cong}\tau_j[\nabla_\varepsilon(e)]$ and we have also a factorization of $h_j$:

\begin{equation}\label{eq-nab}
e_j \xrightarrow{h_j} d_j = e_j \xrightarrow{s_j\alpha_j}\tau_j[\nabla_\varepsilon(e)] \xrightarrow{h_j^{\infty} s_j^{-1}} d_j
\end{equation}

Now both $h: e \to d$ and $\xi \alpha: e \to d$ have the same domain and codomain. Furthermore thanks to \eqref{eq-nab} they have same projections $\tau_j(h)= \tau_j(\xi \alpha)$. By uniqueness of map in the pullback with same (co)domain and same projections we have an equality $h= \xi \alpha$ as desired.

\paragraph{Left perfectness} By construction the universal map $b \to p[\nabla_\varepsilon(e)]$ is essentially the canonical map $b^0 \to b^{\infty}$ which is just the transfinite composite of the maps $\iota^k:b^k \to b^{k+1}$. Now $\iota^k: b^k \to b^{k+1}$ is the canonical map that comes when forming the wide pushout of the maps $\{b^k \xrightarrow{\delta_j^k} p_j(e_j^{k+1}) \}_{j \in \J}$. Therefore if each $b^k \xrightarrow{\delta_j^k} p_j(e_j^{k+1})$ is a trivial cofibration, then so is $\iota^k: b^k \to b^{k+1}$ as well as  every canonical map $p_j(e_j^{k+1}) \xrightarrow{\varepsilon_j^{k+1}} b^{k+1}$. Now by assumption each $p_j$ is left perfect, therefore by induction each map $b^k \xrightarrow{\delta_j^k} p_j(e_j^{k+1})$ is trivial cofibration since $\delta_j^0$ is.
\end{proof}

\begin{rmk}
In the previous proof the map $\eta:e \to \nabla_\varepsilon(e)$ is induced by the family of maps $\alpha_j:e_j \to e_j^{\infty}$ as $j$ varies. In particular we have an isomorphism $\tau_j(\eta) \cong \alpha_j$  in $\E_j^{[1]}$. 

Now recall that for each $j$ the map $\alpha_j:e_j \to e_j^{\infty}$ is the transfinite composite of the structure map $\eta_j^k:e_j^k \to e_j^{k+1} $; and this structure map is, by construction, the displacement of $e_j^k$ along $\varepsilon_j^k: p_j(e_j^k) \to b^k$.
\end{rmk}

A direct consequence of this remark is the following:
\begin{cor}\label{cor-stab-map-intersec-disp}
Let $\D_j$ be a category containing a class of maps $Z_j$ that is closed under transfinite composition, and let $\chi_j: \E_j \to \D_j$ be a functor.  Then with the previous notation and assumptions, if for every $k$ the map $\chi_j(\eta_j^k)$ is in $Z_j$ then the map 
$$\chi_j(\tau_j (\eta))$$
is also in $Z_j$.
\end{cor}
In practice the class $Z_j$ is the class of trivial cofibrations in the model category $\D_j$. And in most cases $\D_j$ will be the category $\Ar(\Aa_j)=\Aa_j^{[1]}$ of morphisms of a model category $\Aa_j$, and the model structure on $\D_j$ is the Reedy (=projective) model structure.  

\subsubsection{Intersection of adjoint functors}

\begin{lem}\label{lem-inter-adj}
Let $\kappa$ be a regular cardinal and let $\{p_j: \E_j \to \B \}_{j \in \J}$ be a $\kappa$-small family of functors over the same base $\B$.  Let $\E= \times_{\B} \E_j$ be a $2$-pullback of this family and let $p: \E \to \B$ be `the' canonical projection.

Assume that
\begin{itemize}
\item Every $p_j$ has a left adjoint $\Gamma_j:  \E_j \to \B$ and all displacements of every $e_j \in \E_j$ exist;
\item Every $p_j$ creates (and hence preserves) filtered colimits in $\E_j$;
\item $\B$ is closed under $\kappa$-small filtered colimits.
\item Assume furthermore that $\B$ is closed under $\kappa$-small wide pushouts and coproducts.
\end{itemize}

Then there is a left adjoint $ \Gamma: \E \to \B$ to $p$. 
\end{lem}
\begin{proof}
This is a corollary of Lemma \ref{lem-displ-intersec}, as we are going to explain. Let $\E'_j= (\1 <\E_j)$ be the join category described in Remark \ref{rmk-ppty-disp}. Let $\B'= (\1 <\B)$ and let $p'_j: \E'_j \to \B'$ be the functor induced by $p_j$. Note that $p'_j$ sends the initial object $0$ of $\E'_j$ to the initial object $0$ of $\B'$. 

Now as we mentioned in Remark \ref{rmk-ppty-disp}, the existence of the left adjoint $\Gamma_j$ is equivalent to the existence of all displacements of the initial object $0 \in \E'_j$ along any map $0 \to b$ in $\B'$. 

Now it's not hard to see that $p'_j$ also creates (and thus preserves) filtered colimits. The existence of coproduct in $\B$ is equivalent to the existence of pushouts of maps $\{0 \to b_i\}$ in $\B'$.  And $\B'$ inherits every pushout that exists in $\B$ (not involving the new object $0$); therefore $\B'$ has all $\kappa$-small pushouts.\\

Now thanks to Corollary  \ref{cor-join-pseudo}, we know that $(1 < \E)$ is equivalent to the pseudopullback of the $p'_i$.\\

We see that we are in the situation of  Lemma \ref{lem-displ-intersec}, and we get that all displacements of  every object $e \in (1 < \E)$ exist. Taking $e=0$ we find a left adjoint to $p: \E \to \B$ as claimed. 
\end{proof}
\section{Pushout, Adjunction and Displacement}

The following lemma will be used to calculate some pushouts in adjunction situations.

\begin{lem}\label{pushout-as-displ}
Let $p: \E \to \B$ be a functor that admits a left adjoint $\Gamma: \B \to \E$. 
Let $\Gamma c \xleftarrow{\Gamma f} \Gamma b \xrightarrow{\sigma} e$ be a pushout data in $\E$ and let $c \xleftarrow{f} b \xrightarrow{\ol{\sigma}} p(e)$ be the pushout data defined by the adjoint-transpose maps.

Let $e \to \Gamma c \cup^{\Gamma b} e $ and  $\varepsilon: p(e) \to [c \cup^{b} p(e)] $ be the canonical maps going to the respective pushout object.

Then $e \to \Gamma c \cup^{\Gamma b} e $ is a displacement of $e$ along $\varepsilon$. In particular we have an isomorphism $ \nabla_\varepsilon(e) \cong \Gamma c \cup^{\Gamma b} e $ in $\E$.
\end{lem}

\begin{proof}
Let $\eta:e \to \nabla_\varepsilon(e)$ be a displacement of $e$ along $\varepsilon: p(e) \to c \cup^{b} e$. There is a universal map $\alpha:[c \cup^{b} p(e)] \to p(\nabla_\varepsilon(e))$ such that $p(e \to \nabla_\varepsilon(e))$ is the following composite.
$$p(e) \xrightarrow{\varepsilon}  [c \cup^{b} p(e)] \xrightarrow{\alpha}p(\nabla_\varepsilon(e)).$$

Using the map $\alpha$, we can extend the universal square obtained from the pushout  $c \xleftarrow{f} b \to p(e)$ to get the following commutative diagram.
\[
\xy
(0,20)*+{b}="X";
(20,20)*+{p(e)}="Y";
(0,0)*+{c}="A";
(20,0)*+{[c \cup^{b} p(e)]}="B";
(50,0)*+{p(\nabla_\varepsilon(e))}="C";
{\ar@{->}^-{\ol{\sigma}}"X";"Y"};
{\ar@{->}_-{\tau_c}"A";"B"};
{\ar@{->}^-{\varepsilon}"Y";"B"};
{\ar@{->}_-{f}"X";"A"};
{\ar@{->}_-{\alpha}"B";"C"};
{\ar@{-->}^-{p(\eta)}"Y";"C"};
\endxy
 \]

The map $c \xrightarrow{\alpha \tau_c} p(\nabla_\varepsilon(e))$ corresponds by adjointness to a unique map $\theta: \Gamma c \to  \nabla_\varepsilon(e)$ in $\E$. Again by adjointness, the above commutative square is equivalent to the (unique) commutative square below; and we are going to show that this is the universal pushout square. 

\[
\xy
(0,20)*+{\Gamma b}="X";
(20,20)*+{e}="Y";
(0,0)*+{\Gamma c}="A";
(20,0)*+{\nabla_\varepsilon(e)}="B";
{\ar@{->}^-{\sigma}"X";"Y"};
{\ar@{->}_-{\theta}"A";"B"};
{\ar@{->}^-{\eta}"Y";"B"};
{\ar@{->}_-{\Gamma f}"X";"A"};
\endxy
 \]

Let $d$ be an object of $\E$ equipped with a co-pushout data $\Gamma c \xrightarrow{\iota_c} d \xleftarrow{ \iota_e} e$ that completes
the pushout data $\Gamma c \xleftarrow{\Gamma f} \Gamma b \xrightarrow{\sigma} e$ into a commutative square. By adjointness there is a unique co-pushout data $ c \xrightarrow{\ol{\iota_c}} p(d) \xleftarrow{p(\iota_e)} p(e)$ that completes the diagram  $c \xleftarrow{f} b \xrightarrow{\ol{\sigma}} p(e)$ into a commutative square. 

The universal property of the later pushout data says that there is a unique map $h: [c \cup^{b} p(e)] \to p(d)$ such that the equalities hereafter hold.
$$ c \xrightarrow{\ol{\iota_c}} p(d)=  c \xrightarrow{\tau_c} [c \cup^{b} p(e)] \xrightarrow{h} p(d);$$
$$ p(e) \xrightarrow{p(\iota_e)} p(d)=  p(e) \xrightarrow{\varepsilon} [c \cup^{b} p(e)] \xrightarrow{h} p(d).$$

If we regard this last equality in the comma category $p(e) \downarrow \B$, we see that $h$ determines a map $[h]: \varepsilon \to p(\iota_e)$.  The displacement along $\varepsilon$ is the adjoint transpose of $\varepsilon$ and so there is a unique map $\xi: \nabla_\varepsilon(e) \to d $ such that the equalities below hold.
\begin{enumerate}
\item $e \xrightarrow{\iota_e} d = e \xrightarrow{\eta} \nabla_\varepsilon(e) \xrightarrow{\xi} d$;
\item $ [c \cup^{b} p(e)] \xrightarrow{h} p(d)=  [c \cup^{b} p(e)] \xrightarrow{\alpha} p(\nabla_\varepsilon(e)) \xrightarrow{p(\xi)} p(d)$. 
\end{enumerate}

The last equality is a factorization of $h$, and we get out of it another equality 
$$c \xrightarrow{\ol{\iota_c}} p(d)= c \xrightarrow{\alpha \tau_c} p(\nabla_\varepsilon(e)) \xrightarrow{p(\xi)} p(d).$$

Now $\ol{\iota_c}$ is the adjoint-transpose map of $\iota_c$ and $\alpha \tau_c$ is the adjoint-transpose map of $\theta$. The naturality of the adjunction $\Gamma \dashv p$ and the uniqueness of the adjoint transpose map say that since we have an equality $\iota_c= p(\xi) \circ (\alpha \tau_c) $ we must have $\iota_c= \xi \theta.$\\

Summing up the above discussion, we find a unique map $\xi : e \to d$ such that 
$$\Gamma c \xrightarrow{\iota_c} d= \Gamma c \xrightarrow{\theta} \nabla_\varepsilon(e) \xrightarrow{\xi} d;$$
$$e \xrightarrow{\iota_e} d= e \xrightarrow{\eta} \nabla_\varepsilon(e) \xrightarrow{\xi} d.$$

This means that the co-pushout data $\Gamma c \xrightarrow{\theta} \nabla_\varepsilon(e) \xleftarrow{\eta} e$ satisfies the universal property of the pushout of $\Gamma c \xleftarrow{\Gamma f} \Gamma b \xrightarrow{\sigma} e$ as claimed. 
\end{proof}

As a corollary we get the following result which is a consequence of the well known transfer result \cite[Theorem 11.3.2]{Hirsch-model-loc}.
\begin{thm}
Let $p: \E \to \B$ be a right adjoint between locally presentable categories whose coefficient category $\B$ is a model category. Then the right-induced model structure exists on $\E$ if and only if the displacement of any object $e \in \E$ along a trivial cofibration $\varepsilon: p(e) \to b$, is a $p$-equivalence.
\end{thm}

\begin{proof}
The if part follows directly from \cite[Theorem 11.3.2]{Hirsch-model-loc} and Lemma \ref{pushout-as-displ}. To get the only if part we proceed as follows. First observe that if the projective model structure exists then $p$ is right Quillen by definition. Equivalently, any left adjoint $\Gamma$ is automatically left Quillen, and therefore preserves trivial cofibrations. And trivial cofibrations are closed under cobase change in the model category $\E$.\\

Given a trivial cofibration $\varepsilon: p(e) \to b$, the displacement of $e$ along $\varepsilon$ is computed thanks to Lemma \ref{pushout-as-displ}, as the following pushout where the attaching map is the co-unit of the adjunction:
$$\Gamma b \xleftarrow{\Gamma \varepsilon} \Gamma p(e) \xrightarrow{} e.$$ 

But since trivial cofibration are closed under cobase change, we get the result.
\end{proof}
\bibliographystyle{plain}
\bibliography{Bibliography_LP_COSEG}

\begin{thebibliography}{1}

\bibitem{Bacard_Sym_QS2}
Hugo Bacard.
\newblock {Symmetries}.
\newblock 7 pages, available at
  \href{http://hal.archives-ouvertes.fr/hal-01018837}{http://hal.archives-ouvertes.fr/hal-01018837},
  July 2014.

\bibitem{Hirsch-model-loc}
Philip~S. Hirschhorn.
\newblock {\em Model categories and their localizations}, volume~99 of {\em
  Mathematical Surveys and Monographs}.
\newblock American Mathematical Society, Providence, RI, 2003.

\bibitem{Joyal_Street_pspb}
Andr{\'e} Joyal and Ross Street.
\newblock Pullbacks equivalent to pseudopullbacks.
\newblock {\em Cahiers Topologie G\'eom. Diff\'erentielle Cat\'eg.},
  34(2):153--156, 1993.

\bibitem{Serre_GAGA}
Jean-Pierre Serre.
\newblock G\'eom\'etrie alg\'ebrique et g\'eom\'etrie analytique.
\newblock {\em Ann. Inst. Fourier, Grenoble}, 6:1--42, 1955--1956.

\bibitem{Toen_Morita}
Bertrand To{\"e}n.
\newblock The homotopy theory of {$dg$}-categories and derived {M}orita theory.
\newblock {\em Invent. Math.}, 167(3):615--667, 2007.

\end{thebibliography}
\end{document}